\title{The Nearest Hermitian Inverse Eigenvalue Problem Solution with Respect to the 2-Norm}

\author{
Marcel Padilla\thanks{ Indian Institute of Technology Mandi, DAAD rise summer internship. Student assistant at the Technical University of Berlin \textit{padilla@math.tu-berlin.de}}\\
\and
Benedikt Kolbe\thanks{Technical University of Berlin}
\and
Aniruddha Chakraborty\thanks{School of Basic Sciences,Indian Institute of Technology Mandi}
}

\documentclass{article}
\usepackage{graphicx}
\usepackage{float}
\usepackage{amssymb}
\usepackage{amsmath}
\usepackage{amsthm}
\usepackage{mathtools}
\usepackage{algorithm}
\usepackage{bbm}
\usepackage[noend]{algpseudocode}
\usepackage{xcolor}
\usepackage{mdframed}

\usepackage{csquotes}
\usepackage{epigraph}
\usepackage{etoolbox}
\makeatletter
\patchcmd{\epigraph}{\@epitext{#1}}{\itshape\@epitext{#1}}{}{}
\makeatother

\usepackage[left=3cm,top=3cm,right=3cm,bottom=3cm,bindingoffset=0.5cm]{geometry}

\usepackage{tikz}
\usetikzlibrary{fadings}
\usetikzlibrary{positioning,calc}
\usetikzlibrary{shapes,arrows}
\tikzstyle{decision} = [diamond, draw, fill=red!50, 
    text width=4.5em, text badly centered, node distance=3cm, inner sep=0pt]
\tikzstyle{block} = [rectangle, draw, fill=blue!20, 
    text width=5em, text centered, rounded corners, minimum height=4em]
\tikzstyle{line} = [draw, -latex']
\tikzstyle{cloud} = [draw, ellipse,fill=red!20, node distance=3cm,
    minimum height=2em]

\usepackage{times}

\newcommand{\paper}{paper {}} 

\newcommand{\Rn}{\mathbb{R}^n}
\newcommand{\Cnn}{\mathbb{C}^{n\times n}}
\newcommand{\Rnn}{\mathbb{R}^{n\times n}}
\newcommand{\reff}[1]{\textbf{(\ref{#1})}}
\newcommand{\ti}[1]{\widetilde{#1}}
\newcommand{\No}{\mathcal{N}}
\newcommand{\ip}{\mbox{ip}}

\theoremstyle{definition}
\newmdtheoremenv[backgroundcolor=gray!20]{theorem}{Theorem}
\AtBeginEnvironment{theorem}{\begin{minipage}{\textwidth}}
\AtEndEnvironment{theorem}{\end{minipage}}

\theoremstyle{definition}
\newmdtheoremenv[backgroundcolor=gray!20]{lemma}[theorem]{Lemma}
\AtBeginEnvironment{lemma}{\begin{minipage}{\textwidth}}
\AtEndEnvironment{lemma}{\end{minipage}}

\AtBeginEnvironment{proposition}{\begin{minipage}{\textwidth}}
\AtEndEnvironment{proposition}{\end{minipage}}

\AtBeginEnvironment{corollary}{\begin{minipage}{\textwidth}}
\AtEndEnvironment{corollary}{\end{minipage}}

 \theoremstyle{definition}
\newmdtheoremenv[backgroundcolor=gray!20]{definition}[theorem]{Definition}
\AtBeginEnvironment{definition}{\begin{minipage}{\textwidth}}
\AtEndEnvironment{definition}{\end{minipage}}

\theoremstyle{definition}
\newmdtheoremenv[backgroundcolor=gray!0]{remark}{Remark}
\AtBeginEnvironment{remark}{\begin{minipage}{\textwidth}}
\AtEndEnvironment{remark}{\end{minipage}}

\begin{document}
\maketitle
\hrulefill
\begin{abstract}
    Assume that the eigenvalues of a finite hermitian linear operator have been deduced accurately but the linear operator itself could not be determined with precision. Given a set of eigenvalues $\lambda$ and a hermitian matrix $M$, this \paper will explain, with proofs, how to find a hermitian matrix $A$ with the desired eigenvalues $\lambda$ that is as close as possible to the given operator $M$ according to the operator 2-norm metric. Furthermore the effects of this solution are put to a test using random matrices and grayscale images which evidently show the smoothing property of eigenvalue corrections.
\end{abstract}



\hrulefill
\section{Introduction}

The aim is to solve the following problem:

\begin{definition}{Nearest Hermitian Inverse Eigenvalue Problem (NHIEP)}

Let $||\ .\ ||$ indicate the operator 2-norm. Given the values $\lambda_1,...,\lambda_n \in \mathbb{C}$ and a hermitian matrix $M\in\Cnn$ find a hermitian matrix $A\in\Cnn$ with eignevalues $\lambda_1,...,\lambda_n$ such that $||A-M||$ is minimal. We call $A$ a solution to the NHIEP of $M$ with eigenvalues $\lambda_1,...,\lambda_n$.

\end{definition}

\quad Inverse eigenvalue problems (IEP) are defined by having a priory partial knowledge of the eigenvalues without having complete information of the linear operator associated to them \cite{Chu98inverseeigenvalue}. Solving an IEP means finding a suitable linear operator under numerous conditions dictated by the application at hand. Important works in this area are excellently collected by Chu, M. and Golub, G. \cite{chu2005inverse}. Every condition placed on the matrix changes the nature of this task and we will focus for a way to compute the optimal solution for the NHIEP. At the end we will discuss possible applications and the level of effect this solution has. 

The core motivation and to solve this problem is it's application in every day numerical computations. Numerical errors in matrix computations alter the eigenvalues in unpredictable ways which will affect ongoing operations, which is critical given that eigenvalues determine important properties and can be very sensitive in algorithms. Important as for example with energy levels in quantum mechanics \cite{weinberg2013lectures} and sensitive as in the Gauss-Seidel iterative solver \cite{demmel1997applied}, when dealing with long term Markov chain behaviour \cite{bremaud2013markov} and whenever it is crucial to keep eigenvalues equal to zero. The numerical errors cause the resulting matrix to remain near the theoretical outcome with high probabilities, which is why with a priory knowledge of the eigenvalues, we expect the solution of the NHIEP to be the best possible guess for theoretical solution because it is characterized by being as close as possible to $M$. This argument also provides reasons to expect the solution of the NHIEP to have a smoothing effect on images.


\begin{center}
\begin{tabular}{| l | c |} 
\hline
\textbf{Purpose} & \textbf{Notation}\\ 
\hline
Fixed dimension & $n\in \mathbb{N}$\\
\hline
Space of n-dimensional hermitian matrices & $H(n)$\\ 
\hline
Space of n-dimensional unitary matrices & $U(n)$\\
\hline
Metric in use & $||A||=\max_{||x||_2=1}||Ax||_2=\sqrt{\rho(A^*A)}$\\
\hline
Conjugate transpose, transpose & $( \ )^* \ \ , \ \ ( \ )^T$ \\
\hline
\end{tabular}
\end{center}

\begin{center}
\begin{tabular}{| c | c | c| c | c | c |} 
\hline
 \textbf{Matrix} & \textbf{Purpose} & \textbf{Eigenvalues} & \textbf{Eigenvectors} & \textbf{Eigen transf.} & \textbf{Diagonalization} \\ 
\hline
 $M\in H(n)$ & Given matrix & $\mu = (\mu_1,...,\mu_n)$ & $w_1,...,w_n$ & $W\in U(n)$ & $D(\mu)\in H(n)$\\ 
\hline
 $A\in H(n)$ & Solution of NHIEP & $\lambda= (\lambda_1,...,\lambda_n)$ & $v_1,...,v_n$ & $V\in U(n)$ & $D(\lambda)\in H(n)$ \\ 
\hline
\end{tabular}
\end{center}

Where $\rho$ refers to the greatest absolute value of the spectrum. We group the set of eigenvalues $\lambda_1,...,\lambda_n$ into a vector $\lambda\in\mathbb{C}^n$ for convenience. We also define:



$$
\Lambda_{\lambda} := \{ A \in \Cnn \ |  A \mbox{ has eigenvalues } \lambda_1,...,\lambda_n \mbox{ with equal multiplicities } \}
$$


\section{Derivation of Solution}\label{sec:dos}

Let us now gather some Lemmas and ideas to help us solve the NHIEP. We write our proofs using many known mathematical theorems found in common text books \cite{IMM2012-03274},\cite{todd1977basic}, \cite{franklin2000matrix} and \cite{fischer2005lineare}.

\begin{lemma}{Existence of Possible Candidates}\label{Lemma:EoPC}

Given $M\in H(n)$ and eigenvalues $\lambda\in \mathbb{C}^n$, a matrix $A\in \Lambda(\lambda)\cap H(n)$ exist if and only if $\lambda\in\Rn$.
\end{lemma}

\begin{proof}
``$\Rightarrow$`` By the spectral theorem \cite{Hawkins19751} we know that any $A\in H(n)$ must have real eigenvalues. Demanding $A$ to have complex eigenvalues is thus impossible. 

``$\Leftarrow$`` $A:=diag(\lambda)\in \Lambda(\lambda)\cap H(n)$ proofs the existence by an example.
\end{proof}

From here on we will assume that our desired eigenvalues $\lambda$ are real. Lemma \reff{Lemma:EoPC} does not prove the existence of a solution of the NHIEP since we do not yet know if a sequence $A_i\in \Lambda(\lambda)\cap H(n)$ minimizing $||A_i-M||$ with increasing $i$ actually converges. However, the existence of a minimum rather than an infimum will become evident with theorem \reff{thr:nhiep}.

By the spectral theorem \cite{Hawkins19751} any hermitian matrix $M$ has a unitary normal basis (UNB) $w_1,...,w_n$ of eigenvectors with real eigenvalues $\{\mu_i\}_{i=1,...,n}$ such that the transformation matrix $W:=(w_1,...,w_n)\in U(n)$ is unitary and $M=W^*diag(\mu)W$. This enables us to parametrize the set $\Lambda(\lambda)\cap H(n)$ on $U(n)$ as seen in the following lemma. 


\begin{lemma}\label{lemma:seq}

If $\lambda\in \mathbb{R}^n \Longrightarrow H(n) \cap \Lambda_{\lambda} = \{ V^*diag(\lambda)V \ | \ V \in U(n) \}$

\end{lemma}

\begin{proof}

``$\subseteq$`` Is given by the spectral theorem.

``$\supseteq$`` We just check the two required properties. $\lambda\in\Rn\Rightarrow (diag(\lambda))^*=diag(\lambda)$, thus making any $V^*diag(\lambda_1,...,\lambda_n)V$ hermitian because 

$$(V^*diag(\lambda)V)^*=(V)^*(diag(\lambda))^*(V^*)^*=V^*diag(\lambda)V \ \ \ \Rightarrow \ \ \ V^*diag(\lambda)V \in H(n)$$

 Let $e_i$ be the $i$-th canonical vector. The eigenvector to the eigenvalue $\lambda_i$ is the column $v_i$ of $V$ as seen in the following calculation:
 
 $$V^*diag(\lambda)Vv_i=V^*diag(\lambda)e_i=V^*\lambda_ie_i=\lambda_iv_i  \ \ \ \Rightarrow \ \ \ V^*diag(\lambda)V \in \Lambda_{\lambda}$$

\end{proof}

Lemma \reff{lemma:seq} shows us that finding the $A\in\Lambda_{\lambda}\cap H(n)$ to minimize $||A-M||$ is equivalent to searching for $V\in U(n)$ to minimize $||V^*D(\lambda)V-M||_m$. Before we are going to use that though we will get rid of some of the non-uniqueness of the diagonanalization and justify this in a lemma.

\begin{lemma}{Eigenvalue Ordering}\label{lemma:eo}

Given $\hat{\mu}\in\Rn, \  M\in \Lambda_{\hat{\mu}}\cap H(n)$ and $\mu=\mbox{sort}(\hat{\mu})$ as the permuted vector such that $\mu_1\leq...\leq\mu_n$. Then there exists a diagonalization $M=W^*D(\mu)W$.

\end{lemma}

\begin{proof}

By lemma \reff{lemma:seq} we can start with a diagonalization $M=\hat{W}^*D(\hat{\mu})\hat{W}$. There exists a unique permutation matrix $P\in\{0,1\}^{n\times n}\cap U(n)$ such that $P\hat{\mu}=\mbox{sort}(\hat{\mu})=\mu$ and thus $P^*\mu=\hat{\mu}$. $D(\mu)P$ permutes the columns and $P^*D(\mu)$ permutes the rows so that together they perform the permutation of the diagonal entries $P^*D(\mu)P=D(\hat{\mu})$. We define $W:=P\hat{W}\in U(n)$ and check if $W^*D(\mu)W=M$.

$$W^*D(\mu)W=(P\hat{W})^*D(\mu)(P\hat{W}) = \hat{W}^*(P^*D(\mu)P)\hat{W} = \hat{W}^*D(\hat{\mu})\hat{W}=M$$

\end{proof}

Lemma \reff{lemma:eo} is useful for standardization as we can now always sort the eigenvalues stored in the vector $\mu,\lambda\in\Rn$ in increasing order. This will later greatly simplify theorem \reff{thr:nhiep}.

But first we need to take another look at the 2-norm. Recall that for hermitian matrices specifically the two norm can be given by the the maximum absolute value of its eigenvalues, precisely $||A||=\rho(A)$.

This also means that the  2-norm only depends on the eigenvalues, thus remaining invariant under transformations that do not change the eigenvalues such as unitary transformations. This is essential for the next lemma to rewrite the NHIEP.

\begin{equation}\label{eq:2n}
\lambda\in\Rn,\ A\in \Lambda_{\lambda}\cap H(n),\ W\in U(n)\Rightarrow
\begin{cases}
\ \ W^*AW\in\Lambda_{\lambda}\cap H(n)\\[3mm]
\ \ ||A||=||W^*AW||\equiv\max_{i=1}^n|\lambda_i|
\end{cases}
\end{equation}

\begin{lemma}{NHIEP Equivalence}\label{lemma:nhiepeq}

The NHIEP is equivalent to finding $T\in U(n)$ to minimize $||T^*diag(\lambda)T - diag(\mu) ||$. A solution can then be expressed as $A:=M+W^*(T^*diag(\lambda)T - diag(\mu))W$.

\end{lemma}

\begin{proof}

Let $M=W^*D(\mu)W$ and $W,V\in U(n)$. Since $U(n)$ is a group \cite{fischer2005lineare} $T:=VW^*$ is also unitary. Thus for each $V\in U(n)$ we can find a unique $T\in U(n)$ such that $V=TW$, meaning also that $\{TW\ | \ T\in U(n)\}$ parametrizes $U(n)$.

We will now use previous lemmata to rewrite the NHIEP step by step. We recall the original form as:

$$ \mbox{Find } A\in H(n) \cap \Lambda_{\lambda} \mbox{ to minimize }  ||A - M || $$

Lets replace $M$ by its diagonalization and use lemma \reff{lemma:seq} to reparametrize $H(n) \cap \Lambda_{\lambda}$ to $\{ V^*diag(\lambda)V \ | \ V \in U(n) \}$.

$$ \mbox{Find } V\in U(n) \mbox{ to minimize }\  ||V^*diag(\lambda)V - W^*diag(\mu)W || $$

Now we apply the reparametrization of $U(n)$ to replace $V$.

$$ \mbox{Find } T\in U(n) \mbox{ to minimize }  ||(TW)^*diag(\lambda)(TW) - W^*diag(\mu)W || = ||W^*( \ T^*diag(\lambda)T - diag(\mu) \ )W ||$$

We will name $Q(T):= T^*diag(\lambda)T - diag(\mu)\in H(n)$. Notice that $||W^*Q(T)W||=||Q(T)||$ due to equation \reff{eq:2n} this transformations preserves the eigenvalues. This means that the solution $T\in U(n)$ to minimize $||Q(T)||$ also reveals a solution to the NHIEP through the equation by rearranging what we just did. 

$$A-M=W^*Q(T)W \Rightarrow A=M+W^*Q(T)W$$

\end{proof}

We can now turn our attention to minimizing $||Q(T)||$ over the set $T\in U(n)$. This is where the order of the eigenvalues plays an important role. Lemma \reff{lemma:eo} guarantees us that we can always arrange them to be monotonically increasing inside the diagonalization. We need to recall Weyl's inequality in matrix theory. 

\begin{theorem}{Weyl's inequality}\label{thr:weyl}

Let $A,B,C\in H(n)$ and $\alpha_1\leq...\leq\alpha_n,\ \beta_1\leq,...,\leq\beta_n,\ \gamma_1\leq...,\leq\gamma_n$ the respective ordered eigenvalues. If $A+B=C$, then for $i,j\in\mathbb{N}$ with $1\leq i\leq n$ and $1\leq i+j-1\leq n$ we have

$$ \alpha_i+\beta_j\leq\gamma_{i+j-1} $$



\end{theorem}

This theorem allows us to proof the following essential inequality:

\begin{theorem}{Lower 2-norm bound}\label{thr:l2nb}

Let $A,B,C\in H(n)$ and $\alpha_1\leq...\leq\alpha_n,\ \beta_1\leq,...,\leq\beta_n$ the ordered real eigenvalues of $A$ and $B$ respectively. If $A-B=C$, then 
$$ \max_{i=1}^n|\alpha_i-\beta_i|\leq||C||$$


\end{theorem}

\begin{proof}

Since the eigenvalues are ordered we know that 

$$||C||_2=\max_{i=1}^n|\gamma_i| = \max(|\gamma_1|,|\gamma_n|)$$

Let $i,j\in\{1,...,n\}$. Weyl's theorem requires the sum matrices, so we define $\hat{B}:=-B$ with eigenvalues $\hat{\beta}_1\leq...\leq\hat{\beta}_n$ and know immediately that $\hat{\beta}_j=-\beta_{n-j+1}$. Applying Weyl's theorem to $A+\hat{B}=C$ setting $j=n-i+1$ returns us:

\begin{equation}\label{eq:lnb1}
 \alpha_i+\hat{\beta}_j\leq\gamma_{i+j-1}\Rightarrow \alpha_i-\beta_i\leq\gamma_{n}
\end{equation}

We can also define $\hat{A}:=-A,\ \hat{C}:=-C$ with eigenvalues $\hat{\alpha}_1\leq...\leq\hat{\alpha}_n,\ \hat{\gamma}_1\leq...,\leq\hat{\gamma}_n$ and apply Weyl's theorem to $\hat{A}+B=\hat{C}$ which is equivalent to $-(A-B)=-C$. Like above, we know that $\hat{\alpha}_i=-\alpha_{n-i+1}$ and $\hat{\gamma}_i=-\gamma_{n-i+1}$. Setting $i=n-j+1$ will lead to:

\begin{equation}\label{eq:lnb2}
\hat{\alpha}_i+\beta_j\leq\hat{\gamma}_{i+j-1}\Rightarrow -\alpha_i+\beta_i\leq-\gamma_{1} \Rightarrow \gamma_1\leq \alpha_i-\beta_i
\end{equation} 

Equation \reff{eq:lnb1} and \reff{eq:lnb2} together give us 

$$ \min(-|\gamma_1|,-|\gamma_n|)\leq |\alpha_i-\beta_i|\leq\max(|\gamma_1|,|\gamma_n|)$$

and thus

$$ \max_{i=1}^n|\alpha_i-\beta_i|\leq\max(|\gamma_1|,|\gamma_n|)=||C||_2 $$


\end{proof}

Theorem \reff{thr:l2nb} provides an important lower bound for us to optimize for the solution of the NHIEP. We have collected enough lemmas and theorems now to resolve the NHIEP with a solution.

\begin{theorem}{NHIEP Solutions}\label{thr:nhiep}

Let $M\in H(n)$ with real eigenvalues $\mu_1\leq...\leq\mu_n$ and $M=W^*D(\mu)W$ for a suitable $W\in U(n)$. Let $\lambda_1\leq...\leq\lambda_n$ be the desired set of real eigenvalues.

Then $A:=W^*D(\lambda)W\in \Lambda_{\lambda}$ is a minimizer of $||A-M||$.

\end{theorem}

\begin{proof}

As mentioned in lemma \reff{lemma:nhiepeq}, the NHIEP can be solved by finding $T\in U(n)$ to minimize the norm of $Q(T):=T^*D(\lambda)T-D(\mu)$. The following equation shows how we can relate this problem to theorem \reff{thr:l2nb}.

\begin{equation}\label{eq:nhiep1}
\underbrace{T^*D(\lambda)T}_{\ti{A}:=}-\underbrace{D(\mu)}_{\ti{B}:=} = \underbrace{Q(T)}_{\ti{C}:=}
\end{equation}

We know instantly that $\ti{A}$ and $\ti{B}$ have the ordered eigenvalues $\lambda_1\leq...\leq\lambda_n$ and $\mu_1\leq...\leq\mu_n$ respectively. Theorem \reff{thr:l2nb} shows us that 

\begin{equation}\label{eq:nhiep2}
\max_{i=1}^n|\lambda_i-\mu_i|\leq||\ti{C}||=||Q(T)||
\end{equation}

Inequality \reff{eq:nhiep2} gives a lower bound for $Q(T)$ meaning that it is impossible to find a $T\in U(n)$ that reduces $||Q(T)||$ further. The very best choice of $T$ we can still hope for is as such that the inequality \reff{eq:nhiep2} turns into an equality. Can we do it? Yes we can! In case of the 2-norm, namely by setting $T=\mbox{Id}_n$ returning $Q(T)=D(\lambda-\mu)$. $||D(\lambda-\mu)||$ is the greatest absolute eigenvalue and thus equal to $\max_{i=1}^n|\lambda_i-\mu_i|$. Note that this simultaneously proves the existence of the minimizer.

Thus by lemma \reff{lemma:nhiepeq} a solution to the NHIEP is

$$A:=M+W^*Q(\mbox{Id}_n)W=W^*D(\mu)W+W^*D(\lambda-\mu)W=W^*D(\lambda)W$$

\end{proof}

Linear maps are not only determined by their entries in a matrix but also by the images of a basis. The NHIEP solution theorem can be interpreted in the following intuitive way as an algorithm that remaps the basis of eigenvectors of $M$.


\epigraph{``Given increasingly sorted real eigenvalues $\{\lambda_i\}_{i=1,...,n}$ and a hermitian matrix $M$, a closest hermitian matrix $A$ from $M$ given the 2-norm with eigenvalues $\{\lambda_i\}_{i=1,...,n}$ is the matrix defined by the mapping that maps a basis of eigenvectors $\{w_i\}_{i=1,...,n}$ of $M$, ordered by their increasing eigenvalues $\{\mu_i\}_{i=1,...,n}$, to $\{\lambda_iw_i\}_{i=1,...,n}$."}

All lemmas and the theorems mentioned in this section are also true if we replace $\Cnn$ by $\Rnn$ throughout the entire setting. Specifically we can replace the hermitian matrices $H(n)$ by symmetric matrices $S(n) \subset H(n)$ and the unitary matrices $U(n)$ with orthogonal matrices $O(n)\subset U(n)$ as every argument used above applies equally well for symmetric matrices with orthogonal matrices. This means that the solution presented here is equally valid for the \textit{nearest symmetric inverse eigenvalue problem}.

\section{Discussion}\label{sec:d}

\quad In this section we will explorer the effects of the solution to the nearest hermitian inverse eigenvalue problem (NHIEP) by looking at some of statistical properties and effects with experiments and visual tests.

\quad Numerical rounding, approximations and data transfer through unreliable channels all cause mistakes in the data we work with. For matrices, eigenvalues determine important properties of a linear operator which is why a small error in an eigenvalue can change critical behaviors of iterative processes. Stochastic processes for example sometimes have critical changes in their development based on their eigenvalues that can distinguish between determined extinction or survival through time, or numerical iterative methods such as the conjugate gradient method's convergence behaviour is determined by the eigenvalues of the input matrix.

Eigenvalues of linear operators need special attention, and therefore we want to use the  NHIEP solution to approximate as good as possible in the 2-norm sens the original hermitian matrix from a faulty one using the correct eigenvalues as input.

The NHIEP solution theorem \reff{thr:nhiep} also tells us that a better solution can not be archived. There might be other matrices that are solutions to the NHIEP but they can not be more optimal and we have no guarantee for the existence of other solutions. To reference this we will name the computation of the solution as a function $\Psi$.

\begin{definition}{NHIEP Solver}

Let $\lambda\in\Rn$ and $M\in H(n)$. The nearest hermitian inverse eigenvalue problem (NHIEP) solution is an an element $A\in H(n)\cap \Lambda_{\lambda}$ that minimizes $||A-M||_2$ and is given by

$$\Psi : H(n) \times \mathbb{R}^n \longrightarrow H(n)$$
$$ ( \ M \ , \ \lambda \ )\mapsto W^*diag(\lambda_1,...,\lambda_n)W$$

where $M\in U(n)$ is as such that  $M=W^*diag(\mu_1,...,\mu_n)W$ with increasingly sorted elements inside $\mu\in\Rn$.

\end{definition}

Note that this function is well defined thanks to the instructions given in theorem \reff{thr:nhiep}.

Algorithm \ref{alg:nhiep} is the outline of an implementation of the function $\Psi$.

\makeatletter
\def\BState{\State\hskip-\ALG@thistlm}
\makeatother

\begin{algorithm}
\caption{Hermitian Inverse Eigenvalue Problem Solver}\label{alg:nhiep}
\begin{algorithmic}[1]
\State \textbf{Input:} Desired eigenvalues $\lambda \in \Rn$, hermitian matrix guess $M \in H(n)$.
\State 
\State Compute $W \in GL(n)$ and $\mu \in \mathbb{R}^n$ by any diagonalization $M=W^*diag(\mu)W$.
\State Orthonormalize the columns of $W=(w_1,...,w_n)$ using the Gram-Schmidt method. 
\State Compute the Permutation Matrix $P$ such that $P\mu$ is sorted by increasing order
\State $W \gets PW$
\State sort $\lambda$ by increasing order
\State $A \gets W^*diag(\lambda)W$
\State 
\State \textbf{Output:} $\Psi(M,\lambda)=A \in H(n)$.
\end{algorithmic}
\end{algorithm}

Note that $W\in GL(n)$ automatically becomes unitary after the orthonormalization. If more than 1 eigenvector has the same eigenvalue the Gram-Schmidt orthonormalization \cite{trefethen1997numerical} will not brake the properties of the columns to be eigenvectors because the eigenspaces are already orthagonal by the spectral theorem.

\subsection{Improvement Ratio}\label{subsec:ipr}

\quad To simplify the implementation we are going to work with the symmetric matricies version of the NHIEP. Let $A=(a_{ij})\in S(n)$ with ordered eigenvalues $\lambda \in \mathbb{R}^n$ be our starting symmetric matrix that we will artificially apply errors to create the matrix $M \in S(n)$ that we will attempt to correct using the NHIEP solver $\Psi$.

The most basic error is to add a scaled normally distributed random number to each entry in A. Let $\No$ be a set of standard normal distributed values. The following table includes all definitions needed to define a measure of improvement experimentally:

\begin{center}\label{table:experiment}
\begin{tabular}{| c | c |}
\hline
\textbf{Symbol} & \textbf{Purpose}\\[2ex] 
\hline
$n\in \mathbb{N}$ & Dimension\\[2ex]
\hline
$A\in S(n),\ a_{ij}\in\No$ & Original matrix\\[2ex]
\hline
$X\in S(n),\ x_{ij}\in\No$ & Distortions matrix\\[2ex]
\hline
$d\in\mathbb{R}_{\geq 0}$ & Error scalar\\[2ex]
\hline
$M=A+dX\in S(n)$ & Distorted matrix\\[2ex]
\hline
$V\in O(n),\ A=V^TD(\lambda)V$ & Diagonalization. $\lambda$ sorted.\\[2ex]
\hline
$W\in O(n),\ M=W^TD(\mu)W$ & Diagonalization. $\mu$ sorted.\\[2ex]
\hline
$\ip(A,M):=1-\dfrac{||A-\Psi(M)||}{||A-M||}$ & Improvement ratio.\\[2ex]
\hline
\end{tabular}
\end{center}

For $d=0$ we get no distortion ($M=A,\ \ip:=0$) and for increasing $d$ the matrix $M$ will become less recognizable as $A$. This will simulate different intensities of errors.


\quad We want to answer the question using the improvement ratio: ``how much closer is the correction $\Psi(M,\lambda)$ to $A$ than the raw distortion $M$ to $A$?.``

For this we will conduct an experiment using Matlab. For each dimension from $n=2,...,20$ and each distortion factor $d\in\{0,25,50,...,175,200\}$ we will generate 1000 symmetric matrices $A\in S(n)$ using normal distributed random values. We then compute each set of eigenvalues $\lambda \in \mathbb{R}$ and distort $A$ to $M$ using $dX$ and compute the \emph{correction} $\Psi(M,\lambda)$.

The decisive quality we want to check is the ratio of distance improvement towards $A$ caused by $\Psi$. An improvement ratio ($\ip$) of 30\% means that the corrected matrix $B$ has moved 30\% closer to the original matrix $A$ relative to $M$.

On the left on Figure \reff{fig:OutputAnalysis} we can see how the correction is more efficient the more it was distorted. Observations involving much higher dimensions hint that the improvement ratio seems to behave similar to $\frac{2}{n+1}$.

\begin{figure}
\centering
\includegraphics[height=4cm]{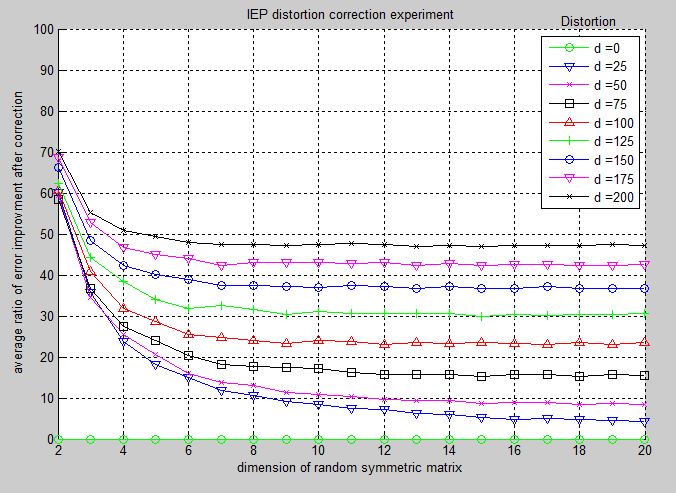}
\includegraphics[height=4cm]{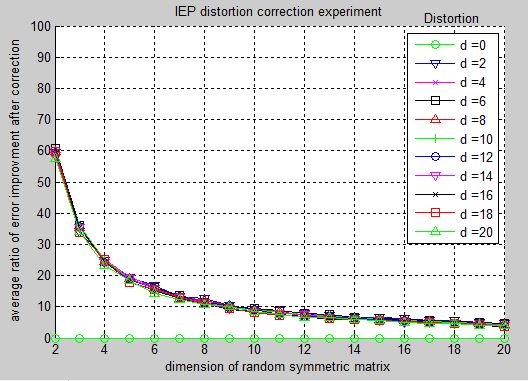}
\label{fig:OutputAnalysis}
\caption{y-axis = average improvement ratio of 1000 samples. x-axis = dimension of square matrix. Left image $d\in\{0,25,...,175,200\}$. Right image $d \in \{0,2,...,18,20\}$.}
\end{figure}

For our practical applications we are interested into smaller distortions which are displayed in the right of Figure \reff{fig:OutputAnalysis} and produce very similar improvement ratios. The decrease in efficiency when the dimensions increase is expected by the fact that the degrees of freedom of $S(n)$ increase quadratically by $\frac{n^2+n}{2}$ while the number of eigenvalues only increases linearly by $n$. With relatively less information at disposal the predictions should be less accurate, and the precise ratio of degrees of freedom is given by $\frac{2}{n+1}$.

\subsection{Image Correction}\label{subsec:ic}

\quad To better visually inspect the effect of the NHIEP solution function $\Psi$ we test it on grayscale images as the are commonly used in image processing \cite{gonzalez2009digital}. We can split any matrix to two symmetric matrices and work on them independently and reunite them at the end. Let $A$ be a grey scaled image matrix with values ranging from 0 to 1. We cut $A$ to two matrices $A_{up},A_{lo}\in S(n)$ by taking the upper and lower triangular matrix of $A$ and reflecting their values to the opposite side respectively. We then compute the sorted eigenvalues $\lambda_{up}\in\Rn$ of $A_{up}$ and $\lambda_{lo}\in\Rn$ of $A_{lo}$ and distort these two matrices to $M_{up},M_{lo} \in S(n)$. $\Psi(M_{up},\lambda_{up})$ and $\Psi(M_{lo},\lambda_{lo})$ are computed as the corrections of $M_{up},M_{lo}$ and are then reunited to form an image again. See figure \reff{fig:imagexperiment} to visually see what we mean to do.

\begin{figure}
\begin{center}

\begin{tikzpicture}[node distance = 2cm, auto, scale=0.60, every node/.style={transform shape}]

    \node [decision] (original) {original matrix \includegraphics[width=0.8cm]{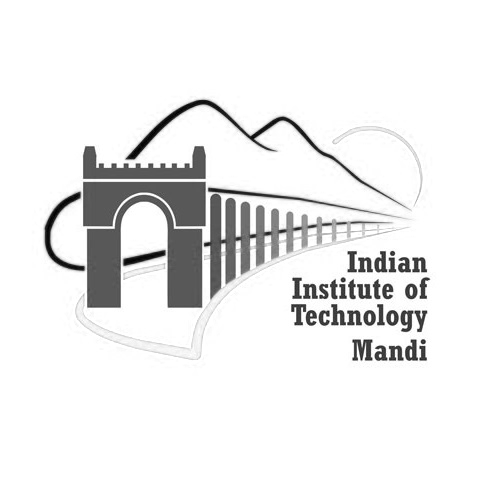}};
    \coordinate [right of=original] (rightoforiginal);
    \node [block, above of=rightoforiginal] (Lo) {\includegraphics[width=0.8cm]{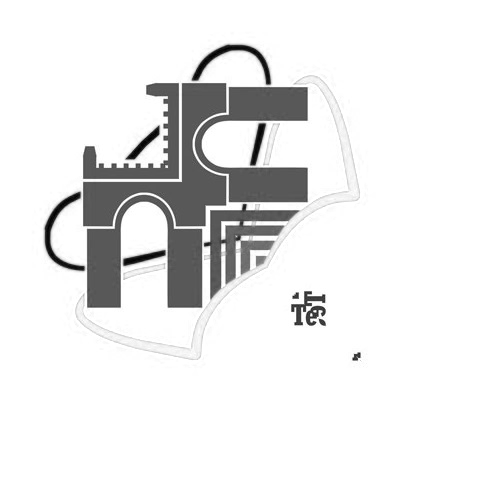}, lower part};
    \node [block, above of=Lo] (Up) {\includegraphics[width=0.8cm]{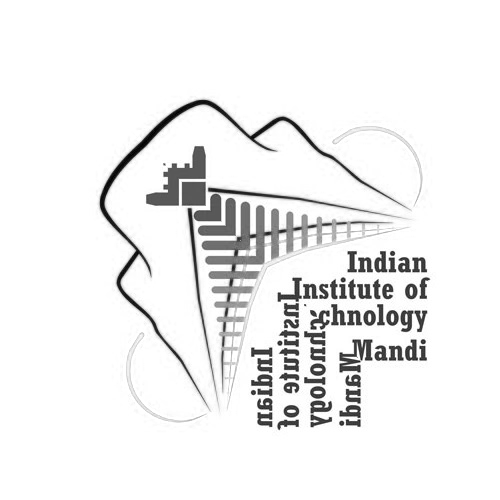} \\ upper part};
    \node [block, right of=Up, node distance=3cm] (eigensendUp) {$\lambda^{[1]}$  \\ eigenvalues};
    \node [block, right of=Lo, node distance=3cm] (eigensendLo) {$\lambda^{[2]}$  \\ eigenvalues};

    \node [block, right of=original, node distance = 10cm] (distorted) { \includegraphics[width=0.8cm]{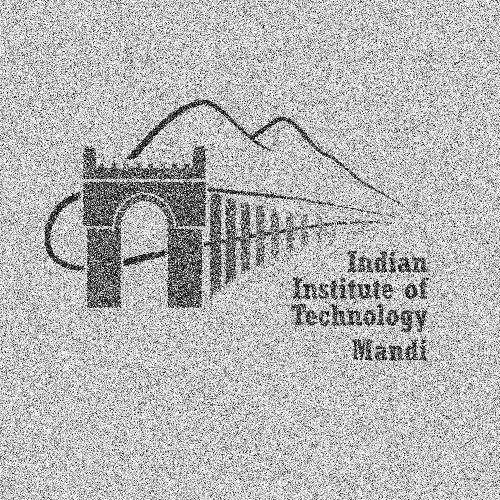} \\ distorted};
    \coordinate [right of=distorted] (rightofdistorted);
    \node [block, above of=distorted, node distance=2cm] (LoDis) {\includegraphics[width=0.8cm,]{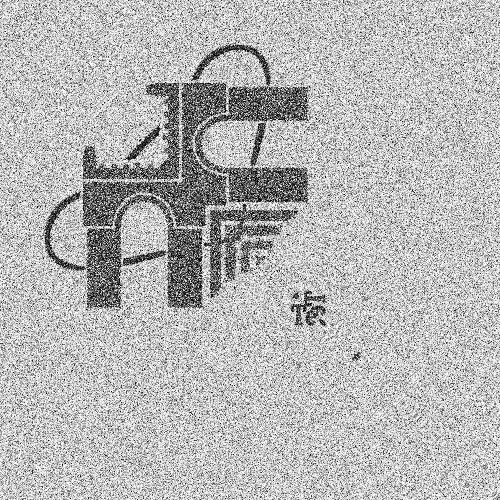}, $\lambda^{[1]}$ \\lower part};
    \node [block, above of=LoDis, node distance=2cm] (UpDis) {\includegraphics[width=0.8cm,]{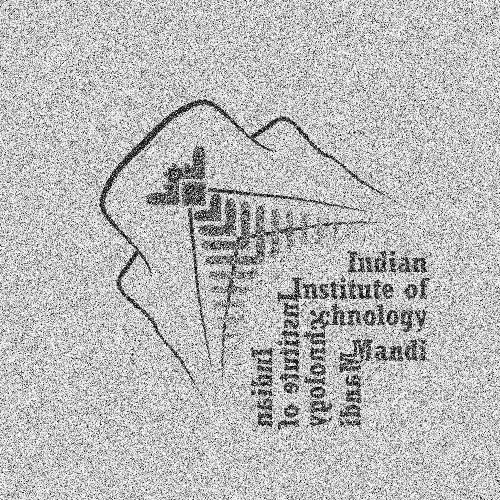}, $\lambda^{[2]}$ \\upper part};
    \node [block, right of=UpDis, node distance=3cm] (UpUnDis) {\includegraphics[width=0.8cm,]{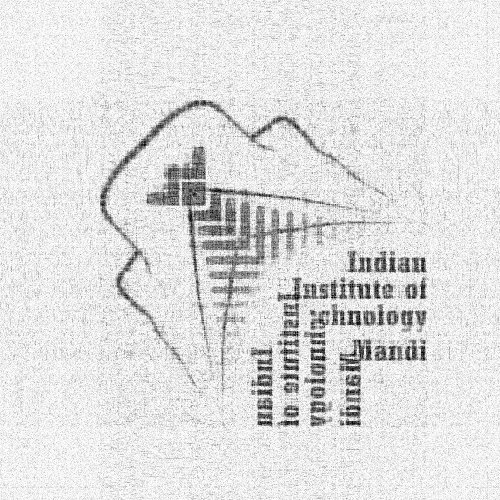}, \\upper correction};
    \node [block, right of=LoDis, node distance=3cm] (LoUnDis) {\includegraphics[width=0.8cm,]{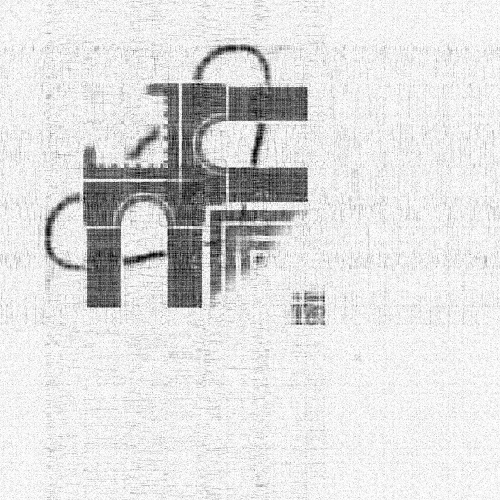}, \\lower correction};
    \node [decision, right of=rightofdistorted] (end) {Corrected matrices merged \includegraphics[width=0.8cm]{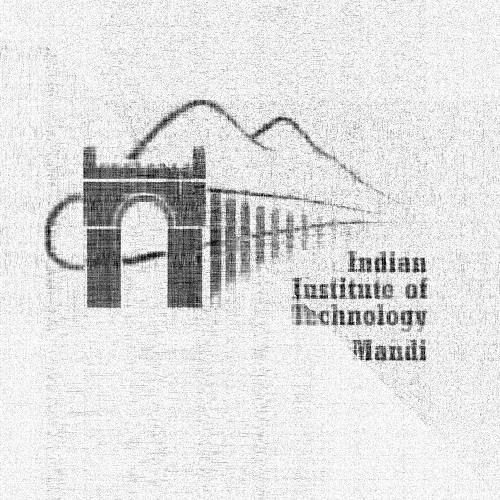}};

    \path [line] (original) |- (Up);
    \path [line] (original) |- (Lo);
    \path [line] (Up) -- (eigensendUp);
    \path [line] (Lo) -- (eigensendLo);
    
    \path [line,dashed] (original) -- node {distortion}(distorted);
    \path [line,dashed] (eigensendUp) -- node {pass on} (UpDis);
    \path [line,dashed] (eigensendLo) -- node {pass on}(LoDis);

    \path [line] (distorted) -- (LoDis);
    \path [line] (LoDis) -- (UpDis);
    \path [line] (UpDis) -- node{$\Phi$}(UpUnDis);
    \path [line] (LoDis) -- node{$\Phi$}(LoUnDis);
    \path [line] (UpUnDis) -| (end);
    \path [line] (LoUnDis) -| (end);
\end{tikzpicture}
\end{center}

\label{fig:imagexperiment}
\caption{Image distrotion experiment set up.}

\end{figure}
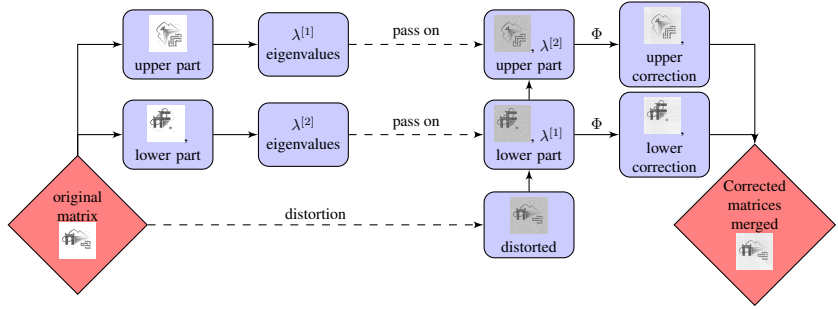

The visualization of the results are made by interpreting values from $0$ to $1$ to $2$ as linear interpolations from black to white and back to black while continuing this periodically.

Figure \ref{fig:iitDistortion} shows two types of images with distinct resolutions being distorted by increased values of $d\in \{0, 10, 20, 30 ,40\}$. On the left side you see the raw distortion of the original images and on the right side is the correction made only through the knowledge of the eigenvalues from the original images. Both examples evidently show us why eigenvalues are a very important degrees of freedom in symmetric matrices. Even though they make up only $n$ degrees from the $\frac{n^2+n}{2}$ degrees of freedom of a symmetric matrix, they visually have great influence on smoothing areas of similar colors and making edges more apparent.
	
	This means that if we regard the image as an operator and the noise as numerical errors we can expect a similar correction to happen by applying the NHIEP solution. The NHIEP solution thus has the natural property of reducing noise and errors. The mathematical reason for this is the nature of the noise being independent and unbiased on each pixel and the attempt of our algorithm to undo this noise.

\begin{figure}
\centering
\includegraphics[width=6cm]{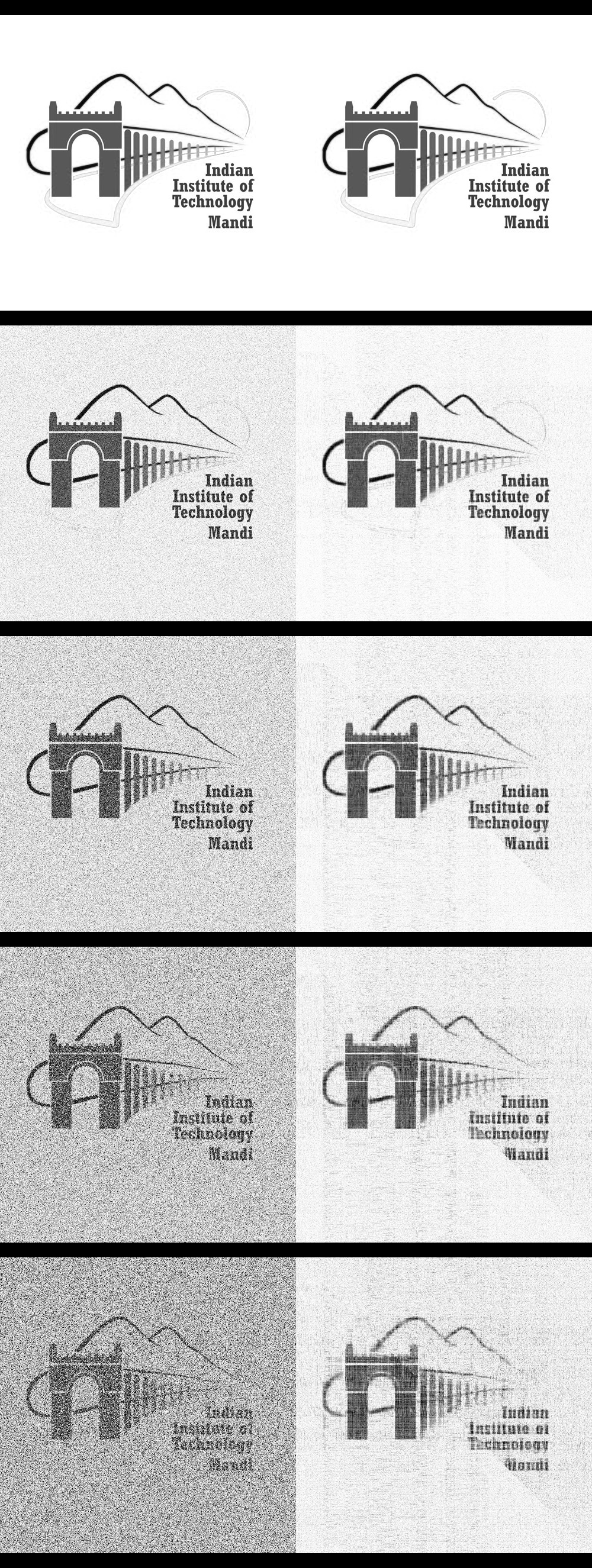}
\includegraphics[width=6cm]{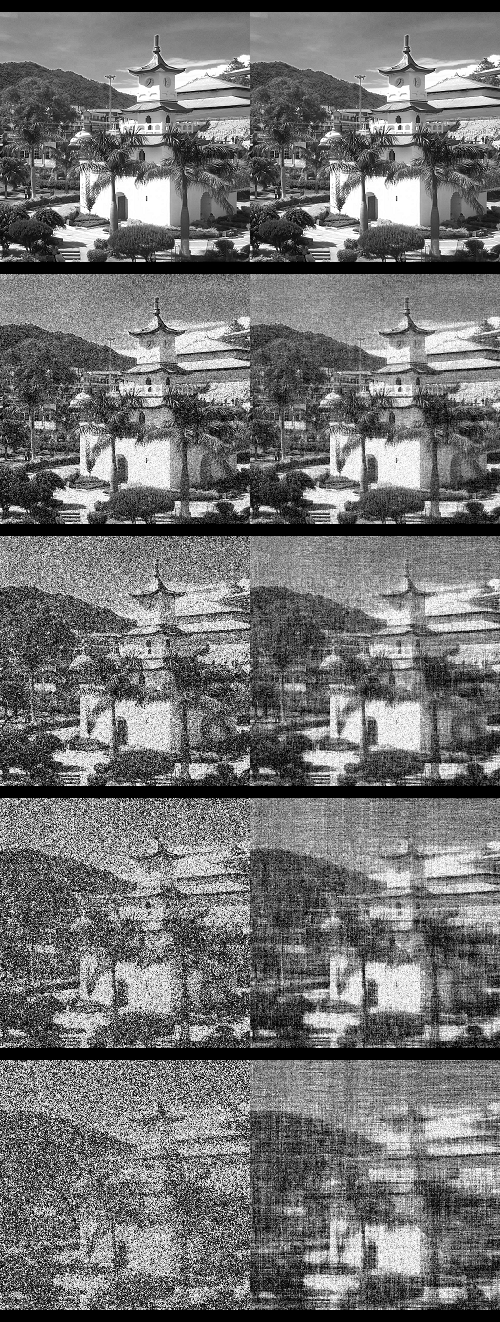}
\label{fig:iitDistortion}
\caption{Left: Logo of the Indian Institute of Technology Mandi. 500x500 pixel image. Right: Central market place in Mandi. 250x250 pixel image. From top to bottom, increasing distortions d = 0, 10, 20, 30 ,40. left side distorted image $M$, right side corrected image $\Phi(M,\lambda)$.}
\end{figure}

\section{Conclusion and Future Ideas}

\quad We have solved the nearest hermitian inverse eigenvalue problem (NHIEP) as fine as possible and seen an algorithm that returns a solution for any hermitian matrix $M\in H(n)$ and $\lambda\in\Rn$. We proved it's right fullness and interpreted its action in section \reff{sec:dos}.

The main theorem number \reff{thr:nhiep} happens to be just as applicable to nearest symmetric inverse eigenvalue problems as to NHIEP. We have also shown how very useful the sorting of eigenvalues can be inside the diagonalization to simplify the computation of the solution. With theorem \reff{thr:l2nb} we also proved an easily comprehensible lower bound for the 2-norm of a difference between hermitian matricies.

The discussion of section \reff{sec:d} acknowledged some of the effects of the NHIEP solution on symmetric matrices. In subsection \reff{subsec:ipr} we have experimentally assessed the average improvement ratio under normal distortions and how that effect decreases with the degree of freedom ratio and in section \reff{subsec:ic} we opened an insight on how the distortion of eigenvalues can affect the overall matrix of a grey scale image and thus observed a natural noise reducing effect given by eigenvalues.

The aim of this \paper was to establish this idea together with a proof and to give an impulse of what can be archived with inverse eigenvalues problems. It expands the way we think of eigenvalues and may perhaps spark a new idea on how to use this in applications.

One future task to be archived is the inclusion of the $1$- $\infty$- and the Frobenius-norm solutions by creating other functions like $\Psi$.

\bibliographystyle{amsalpha}
\bibliography{Ref}

\end{document}